\documentclass[runningheads, envcountsame, a4paper]{llncs}

\usepackage{amsmath}
\usepackage{amssymb}
\usepackage[utf8]{inputenc}

\usepackage{hyperref}

\def\N{\mathbb N}
\def\A{\mathcal A}
\def\C{\mathcal C}
\def\pf{\begin{proof}}
\def\pfk{\end{proof}}

\def\N{\mathbb N}

\def\A{\mathcal A}
\def\C{\mathcal C}

\def\B{{\mathcal B}}
\def\Pal{{\mathrm{Pal}}}
\def\Pc{{\mathcal P}}

\def\buu{{\mathbf u}}
\def\bff{{\mathbf f}}
\def\btt{{\mathbf t}}
\def\bvv{{\mathbf v}}
\def\bww{{\mathbf w}}

\renewcommand{\L}[1][]{\mathcal{L}_{#1}(\mathbf{u})}

\def \L {\mathcal{L}}
\def \Lu {\L(\buu)}
\def \Rk#1 {$\mathcal{R}_{#1}$}

\DeclareMathOperator{\lpp}{lpp}
\DeclareMathOperator{\lps}{lps}

\def \FC#1 {
\mathcal{C}
\ifthenelse{\equal{#1}{}}{}{(#1)}
}

\def \PC#1 {
\mathcal{P}
\ifthenelse{\equal{#1}{}}{}{(#1)}
}

\begin{document}
\title{On Words with the Zero Palindromic Defect}
\titlerunning{On Words with the Zero Palindromic Defect}
\toctitle{On Words with the Zero Palindromic Defect}

\author{Edita Pelantov\'a\inst{1} \and \v St\v ep\'an Starosta\inst{2}}
\authorrunning{E. Pelantov\'a and \v St\v ep\'an Starosta}
\tocauthor{Edita~Pelantov\'a\ and \v St\v ep\'an~Starosta}

\institute{Department of Mathematics, Faculty of Nuclear Sciences and Physical Engineering, Czech Technical University in Prague, Czech Republic\\ \email{edita.pelantova@fjfi.cvut.cz} \and Department of Applied Mathematics,Faculty of Information Technology, Czech Technical University in Prague, Czech Republic\\ \email{stepan.starosta@fit.cvut.cz}}

\maketitle
\setcounter{footnote}{0}

\begin{abstract}
We study the set of finite  words with zero palindromic defect, i.e., words rich in palindromes.
This set is factorial, but not recurrent.
We focus on description of pairs of  rich words which cannot occur simultaneously as factors of a longer rich word.
\end{abstract}

\section{Introduction}

In \cite{DrJuPi},  Droubay, Justin and Pirillo  observed  that the  number of distinct palindromes occurring in a  finite word $w$  of length $n$  does not exceed $n+1$.  This upper bound motivated Brlek, Hamel,  Nivat,  and Reutenauer to define in \cite{BrHaNiRe} the notion \textit{palindromic defect}  $D(w)$ of a finite word  $w$ as the difference of the upper bound $n+1$ and the actual number of palindromic factors   occurring in  $w$. One can say that the  palindromic defect  measures the number of ``missing'' palindromic factors in the given word.
A word with zero palindromic defect is usually shortly called {\em rich} or {\em full}.

For an infinite word $\bf{u}$  the palindromic defect   $D(\bf{u})$  is naturally defined as the supremum of the set $\{D(w) \colon w \text{ is a factor of } \bf{u}\} $. Many classes of  words with the defect zero have been found, for example  Sturmian words, words coding symmetrical interval exchange and complementary symmetric Rote words (see~\cite{BaMaPe,BlBrLaVu11,DrPi}).

Palindromic defect is actively studied in the last decade.
During these years many nice properties of  words with zero defect have been brought into light. Some  of them have been already proved, some of them are  formulated as conjectures and  are still open.    Neither the basic question ``What is the number of rich words of a given length?''  has been  answered.
This question is extremely  interesting  as  the set of rich words is a very naturally defined factorial language  which has  superpolynomial and subexponential  growth as was shown in \cite{GuShSh15}  by C. Guo, J. Shallit and A. M.
Shur  and in  \cite{Ru17} by J. Rukavicka, respectively.

This article consists of three parts.
In the first part, we present relevant known results.
In the last part we give a list of open questions connected to the palindromic defect  and we also recall a narrow connection to the well known conjecture of Hof, Knill, and Simon.
The middle part contains a new result.
It  is devoted to so-called  compatible  words, i.e.,  to the pairs of finite rich words which can occur simultaneously as factors of a longer rich word.
We believe  that our  result  may help to characterize  words $w$ with the following property:   $D(w) = 1$  and $D(u) =0$ for each proper factor $u$ of $w$.
A characterization of these words seems to be the  missing point in answering several open questions.

\section{Preliminaries}

\subsection{Basic notations and definitions}

Let $\A$ be a finite set, called an \textit{alphabet}.
Its elements are called \textit{letters}.
A \textit{finite word} $w$ is an element of $\A^n$ for $n \in \N$.
The \textit{length} of $w$ is $n$ and is denoted $|w|$.
The set of all finite words over $\A$ is denoted $\A^*$.
An \textit{infinite word} over $\A$ is an infinite sequence of letters from $\A$.

A finite word $w$ is a \textit{factor} of a finite or infinite word $v$ if there exist words $p$ and $s$ such that $v$ is a concatenation of $p$, $w$, and $s$, denoted $v = pws$.
The word $p$ is said to be a \textit{prefix} and $s$ a \textit{suffix} of $v$.
The set of all factors of a word $\buu$ is the \textit{language of ${\bf{u}}$} and is denoted $\Lu$.
All factors of $\buu$ of length $n$ are denoted by $\L_n(\buu)$.

An \textit{occurrence} of $w = w_0w_1 \cdots w_{n-1} \in \A^n$ in a word $v = v_0v_1v_2 \ldots$ is an index $i$ such that $v_i \cdots v_{i+n-1} = w$.
A factor $w$ is \textit{unioccurrent} in $v$ if there is exactly one occurrence of $w$ in $v$.
A \textit{complete return word} of a factor $w$ (in $v$) is a factor $f$ (of $v$) containing exactly two occurrences of $w$ such that $w$ is its prefix and also its suffix.
For instance, the word $010011010$ is a complete return word of $010$.

The \textit{reversal} or mirror mapping assigns to a word $w \in \A^*$ the word $\widetilde{w}$ with the letters reversed, i.e.,
\[
\widetilde{w} = w_{n-1}w_{n-2}\cdots w_1 w_0 \quad \text{ where } w = w_0w_1 \cdots w_{n-1} \in \A^n.
\]
A word is \textit{palindrome} if $w = \widetilde{w}$.
We say that a language $\L \subset \A^*$ is \textit{closed under reversal} if for all $w \in \L$ we have $\widetilde{w} \in \L$.

Given an infinite word $\buu$, its \textit{factor complexity} $\C_\buu(n)$ is the count of its factors of length $n$:
\[
\C_\buu(n) = \# \L_n(\buu) \quad \text{ for all } n \in \N.
\]
Let $\Pal(\buu)$ be the set of all palindromic factors of the infinite word $\buu$.
The \textit{palindromic complexity} $\Pc_\buu(n)$ of $\buu$ is given by
\[
\Pc_\buu(n) = \# ( \L_n(\buu) \cap \Pc(\buu) ) \quad \text{ for all } n \in \N.
\]
We omit the subscript $\buu$ if there is no confusion.

 \subsection{Fixed points of morphisms and their properties}

A \textit{morphism} $\varphi$ is a mapping $\A^* \to \B^*$ where $\A$ and $\B$ are alphabets such that for all $v,w \in \A^*$ we have $\varphi(vw) = \varphi(v)\varphi(w)$ (it is a homomorphism of the monoids $\A^*$ and $\B^*$).
Its action is extended to $\A^\N$: if $\buu = u_0u_1u_2 \ldots \in \A^\N$, then
\[
\varphi(\buu) = \varphi(u_0) \varphi(u_1) \varphi(u_2) \ldots \in \B^\N.
\]

If $\varphi$ is an endomorphism of $\A^*$, we may find its fixed point, i.e., a word $\buu$ such that $\varphi(\buu) = \buu$.
We are interested mainly in the case of $\buu$ being infinite.
A morphism $\varphi: \A^* \to \A^*$ is \textit{primitive} if there exists an integer $k$ such that for every $a,b \in \A$ the letter $b$ occurs in $\varphi^k(a)$.

Two morphisms $\varphi, \psi: \A^* \to \B^*$ are \textit{conjugate} if there exists a word $w \in \B^*$ such that
\[
\forall a \in \A, \varphi(a)w = w \psi(a) \quad  \text{ or } \quad \forall a \in \A, w\varphi(a) = \psi(a)w.
\]
If $\varphi$ is primitive, then the languages of fixed points of $\varphi$ and $\psi$ are the same.

A morphism $\psi: \A^* \to \B^*$ is of \emph{class $P$} if $\psi(a) = pp_a$ for all $a
\in \A$ where $p$ and $p_a$ are both palindromes (possibly empty).
A morphism $\varphi$ is of class $P'$ if it is conjugate to a morphism of class $P$.

The following examples illustrate the last few notions.

\begin{example}
Let $\varphi: \{a,b\}^* \to \{a,b\}^*$ be determined by $
\varphi: \begin{array}{rcl}
a & \mapsto & abab, \\
b & \mapsto & aab.
\end{array}
$
The fixed point of $\varphi$ is
\[
\buu = \lim_{k \to +\infty} \varphi^k(a) = \underbrace{abab}_{\varphi(a)} \underbrace{aab}_{\varphi(b)} \underbrace{abab}_{\varphi(a)} \underbrace{aab}_{\varphi(b)} \underbrace{abab}_{\varphi(a)} \ldots
\]
The morphism $\varphi$ is of class $P'$ since it is conjugate to $\psi$ given by
$
\psi: \begin{array}{rcl}
a & \mapsto & abab, \\
b & \mapsto & aba.
\end{array}
$
Indeed, we have $ab \varphi(a) = \psi(a) ab$ and $ab \varphi(b) = \psi(b) ab $.
To see that $\psi$ is of class $P$, i.e., it is of the form $a \mapsto pp_a$ and $b \mapsto pp_b$, it suffices to set $p = aba$, $p_a = b$ and $p_b = \varepsilon$.
The fixed point of $\psi$ is
\[
\bvv = \lim_{k \to +\infty} \psi^k(a) = \underbrace{abab}_{\psi(a)} \underbrace{aba}_{\psi(b)} \underbrace{abab}_{\psi(a)} \underbrace{aba}_{\psi(b)} \underbrace{abab}_{\psi(a)} \ldots
\]
We have $\L(\buu) = \L(\bvv)$.

\end{example}

\begin{example} \label{ex:TM_Fibo}
The two  famous examples of infinite words, the Thue--Morse word $\btt$ and the Fibonacci word $\bff$, are both fixed points of a morphism.

The word $\btt$ is fixed by the morphism $\varphi_{TM}$ determined by $\varphi_{TM}(0) = 01$ and $\varphi_{TM}(1) = 10$.
Note that this morphism in fact has two fixed points, one being the other one after replacing $0$ with $1$ and $1$ with $0$.

The word $\bff$ is fixed by the morphism $\varphi_F$ defined by $\varphi_F(0) = 01$ and $\varphi_F(1) = 0$.
\end{example}

An (infinite) fixed point of a morphism of class $P'$ clearly contains infinitely many palindromes which is one motivation for this notion.
Class $P$ is introduced in \cite{HoKnSi} in the context of discrete Schr\"{o}dinger operators.

\section{The study of palindromic defect}

\subsection{Characterizations of words with the zero  defect}

We start by giving some of the known characterizations of infinite rich words.

\begin{theorem} \label{thm:equiv_almost_rich}
For an infinite word $\buu$ with language closed under reversal
the following statements are equivalent:
\begin{enumerate}
\item  \label{equiv_arich_1} $D(\buu)$ is zero (\cite{BrHaNiRe});
\item \label{equiv_arich_psuffix} any prefix of $\buu$  has a unioccurrent longest palindromic suffix (\cite{DrJuPi});
\item \label{equiv_arich_returnpalindromicky} for any palindromic factor $w$ of $\buu$, every complete return word of $w$ is a~palindrome (\cite{GlJuWiZa});
 \item \label{equiv_arich_returnobecny}  for any factor $w$ of $\buu$, every factor of $\buu$ that contains $w$ only as its prefix and $\widetilde{w}$ only as its suffix is a~palindrome (\cite{GlJuWiZa});
 \item \label{equiv_arich_nasa} for
each $n \in N$ we have $\C(n+1) - \C(n) + 2 =
\Pc(n) + \Pc (n+1)$ (\cite{BuLuGlZa}).
\end{enumerate}
\end{theorem}

We  generalized  the previous theorem to  infinite  words with finite palindromic defect, see \cite{BaPeSta3,PeSta1}.
 In particular, we showed that an infinite word has a finite palindromic defect $D(\buu)$ if and only if the equality  $\C(n+1) - \C(n) + 2 =\Pc(n) + \Pc (n+1)$ is valid for all $n\in N$ up to finitely many exceptions.   A surprising observation that  these  exceptional indices allow to determine the value of the palindromic defect  was made by Brlek and Reutenauer. In \cite{BrRe-conjecture} they proved for infinite periodic words and  conjectured for general words the following equality
\begin{equation}\label{oni}
2 D(\buu) = \sum_{n = 0}^{+\infty} \Bigl( \C_\buu(n+1)-\C_\buu(n)+2-\Pc_\buu(n+1)-\Pc_\buu(n) \Bigr).
\end{equation}
The conjecture was  confirmed in \cite{BaPeSta5} where we showed the following theorem.

\begin{theorem} \label{co:BrRe}
Equation \eqref{oni} is true for any infinite word  $\buu$ whose language is closed under reversal.
\end{theorem}

Besides these general properties, many examples of words with zero or finite palindromic defect were found:
\begin{itemize}
  \item In \cite{BuLuGlZa2,PeSta2}, another characterizations of rich words are given.
  \item In \cite{BuLuLu}, the relation of rich words to so-called periodic-like words is exhibited.
  \item Links to another class of words, trapezoidal words, are shown in \cite{LuGlZa_tr}.

  \item Words coding symmetric interval exchange transformations are rich by \cite{BaMaPe}.
  \item In \cite{BlBrLaVu11}, the authors show that words coding rotation on the unit circle with respect to partition consisting of two intervals are rich.
  \item In \cite{ReRo}, the authors show a connection of rich words with the Burrows--Wheeler transform.
  \item In \cite{Sta2015}, we show that morphic images of episturmian words, a known class of rich words, produces a word with finite palindromic defect.
  \item The articles \cite{JaPeSta1,PeSta3,Sta2011} exhibit more examples of words with finite palindromic defect (along with some examples of words with finite generalized palindromic defect).
\end{itemize}

\subsection{Palindromic defect of fixed points of morphisms}

We now focus on words that are fixed by a morphism with the assumption that their language is closed under reversal.
The main motivation to study their palindromic defect is the following conjecture.

\begin{conjecture}[Zero defect conjecture \cite{BlBrGaLa}] \label{co:defect}
Let $\buu$ be an aperiodic fixed point of a primitive morphism having its language closed under reversal.
We have $D(\buu) = 0$ or $D(\buu) = +\infty$.
\end{conjecture}

The Thue--Morse word $\btt$ and the Fibonacci word $\bff$ are examples of aperiodic fixed points of a primitive morphism (see Example~\ref{ex:TM_Fibo}) having their language closed under reversal.
We have $D(\bff) = 0$ and $D(\btt) = +\infty$.

Counterexamples to the conjecture were given in \cite{Basic13,BuVa12}.
Thus, the current statement of the conjecture is not true.
There still might some refinement of the current statement that is valid as there are many witnesses and the found counterexamples seem to have some specific properties.
Indeed, in \cite{LaPeSta1} we prove that the conjecture is true for a special class of morphisms.
A morphism $\varphi$ is \textit{marked} if there exists two morphisms $\varphi_1$ and $\varphi_2$, both being conjugate to $\varphi$,
such that \[
\{ \text{last letter of } \varphi_1(a) \colon  a \in \A \} = \{ \text{first letter of } \varphi_2(a) \colon  a \in \A \} = \A.
\]
In other words, the set of the last letters of the images of letters by $\varphi_1$ is the whole alphabet $\A$ and the set of the first letters of the images of letters by $\varphi_2$ is also the whole alphabet $\A$.

For instance, $\varphi = \varphi_{TM}: 0 \mapsto 01, 1 \mapsto 10$ is marked (here $\varphi = \varphi_1 = \varphi_2$).
For $\varphi = \varphi_F: 0 \mapsto 01, 1 \mapsto 0$ we have $\varphi = \varphi_1$ and $\varphi_2: 0 \mapsto 10, 1 \mapsto 0$.
Thus, $\varphi_F$ is also marked.

If a morphism $\varphi$ is conjugate to no other morphism except for $\varphi$ itself, then we say that $\varphi$ is \emph{stationary}.
In other words, a morphism $\varphi$ is stationary if the longest common prefix and the longest common suffix of $\varphi$-images of all letters are both empty words.

In \cite{LaPeSta1} we show the following theorems:
\begin{theorem} \label{th:0dc_main}
Let $\varphi$ be a primitive marked morphism and let $\buu$ be its fixed point with finite palindromic defect.
If all complete return words of all letters in $\buu$ are palindromes or $\varphi$ is not stationary, then $D(\buu) = 0$.
\end{theorem}

Moreover, the binary alphabet allows for all of the assumptions to be dropped:

\begin{theorem} \label{th:ZDCbinary}
If $\buu \in \A^\N$ is a fixed point of a primitive morphism over binary alphabet and $D(\buu) < +\infty$, then $D(\buu) = 0$ or $\buu$ is periodic.
\end{theorem}

We thus confirm that for a large class of fixed points of morphisms, their palindromic defect is either zero or infinite.

\subsection{Enumeration of Rich Words}

Let $R_{d}(n)$ denote the number of rich words of length $n$ over an alphabet with $d$ elements.
As we have already mentioned, there is no closed-form formula for $R_{d}(n)$.

 In \cite{Vesti2014}, Vesti gives a recursive lower
bound on $R_{d}(n)$  and an upper bound on $R_{2}(n)$.  Both these estimates seem to be very rough.

 In
\cite{GuShSh15}, Guo, Shallit and Shur constructed for each $n$ a large set of binary rich
words of length $n$.
They show that for any two sequences of integers $0\leq n_1\leq n_2 \leq \cdots \leq n_k$ and $0\leq m_1\leq m_2 \leq \cdots \leq m_k$  satisfying  $n= \sum_{i=1}^k n_k +  \sum_{i=1}^k m_k$, the word $a^{n_1}b^{m_1}a^{n_1}b^{m_1}\cdots a^{n_k}b^{m_k}$ of length $n$ is rich.
This construction gives, currently, the best lower bound
on the number of binary rich words, namely $R_2(n) \geq \frac{ C^{\sqrt{n}}}{p(n)}$ where $ p(n)$ is a
polynomial and the constant $C \sim 37$.   They also conjectured that  $R_2(n) = \Theta\Bigl(\frac{n}{g(n)}\Bigr)^{\sqrt{n}}$ for some
infinitely growing function $g(n)$.

The best upper bound is provided by Rukavicka in \cite{Ru17}. He shows that   $R_d(n)$ has a subexponential
growth on any alphabet.  More precisely, for any cardinality $d$ of the alphabet  $\lim\limits_{n\to \infty}\sqrt[n]{R_d(n)}=1$.
The result uses a specific factorization of a rich word into
distinct rich palindromes, called UPS-factorization (Unioccurrent Palindromic
Suffix factorization).

\section{Compatible Pairs}

The set of rich words is a factorial language but it is not recurrent.
Let us recall that a language $\mathcal{L} \subset \mathcal{A}^*$ is {\em recurrent} if for any two words $u,v \in  \mathcal{L}$ there exists $w \in \mathcal{L}$ such that $u$ is a prefix of $w$ and $v$ is a suffix of $w$.
Using  results of Glen et al. \cite{GlJuWiZa}, Vesti in~\cite{Vesti2014} formulated a sufficient condition which prevents two rich words $u,v$ to be simultaneously  factors of another rich word.
His proposition uses the notion of {\em longest palindromic suffix} of a factor $u$, denoted $\lps(u)$ and  {\em longest palindromic prefix} of a factor $u$, denoted $\lpp(u)$.
We say that two finite words are {\em compatible} if there exists a rich word having these two words as factors.

\begin{proposition}\label{vesti}  Let   $u$ and $v$ be two words such that
\begin{equation}\label{E1} u\neq  v, \ \ u,v \ \  \text{rich}, \quad    \lpp(u)=\lpp(v) \quad \text{and}\quad  \lps(u)=\lps(v).
\end{equation} If a word $w$ contains factors $u$ and $v$,  then $w$ is not rich, i.e.,    $u$ and $v$  are not compatible.
\end{proposition}

We give an example which demonstrates that a word $w$ can be non-rich without  containing factors $u$ and $v$ satisfying \eqref{E1}.
\begin{example} \label{ex:exx} Consider the word $w=11010011$, which is not rich. In fact, it is a factor of the Thue--Morse word.
As pointed out in \cite{BlBrFrLaRi},  the length 8 is the shortest length of a non-rich binary word.

Table~\ref{ta:exx} depicts all non-empty rich factors $u$ of $w$ together with the pairs $(\lpp(u), \lps(u))$.
The map $u \mapsto (\lpp(u), \lps(u))$ is injective.    In other words, no pair of factors $u,v$ of the non-rich word $w=11010011$  satisfies \eqref{E1}.

\begin{table}

\begin{center}
\begin{minipage}[t]{0.5\textwidth} \centering
\begin{tabular}{cc}
$u$ & $(\lpp(u), \lps(u))$ \\ \hline
$1$ & $(1,1)$ \\
$11$ & $(11,11)$ \\

$110$ & $(11,0)$ \\

$1101$ & $(11,101)$ \\

$11010$ & $(11,010)$ \\

$110100$ & $(11,00)$ \\

$1101001$ & $(11,1001)$ \\

$10$ & $(1,0)$ \\

$101$ & $(101,101)$ \\

$1010$ & $(101,010)$ \\

$1010$ & $(101,010)$ \\

$10100$ & $(101,00)$ \\

$101001$ & $(101,1001)$ \\

$1010011$ & $(101,11)$ \\

\end{tabular}
\end{minipage}
\begin{minipage}[t]{0.5\textwidth}  \centering
\begin{tabular}{cc}
$u$ & $(\lpp(u), \lps(u))$ \\ \hline

$0$ & $(0,0)$ \\

$01$ & $(0,1)$ \\

$010$ & $(010,010)$ \\

$0100$ & $(010,00)$ \\

$01001$ & $(010,1001)$ \\

$010011$ & $(010,11)$ \\

$100$ & $(1,00)$ \\

$1001$ & $(1001,1001)$ \\

$10011$ & $(1001,11)$ \\

$00$ & $(00,00)$ \\

$001$ & $(00,1)$ \\

$0011$ & $(00,11)$ \\

$011$ & $(0,11)$ \\
& \\ 
\end{tabular}

\end{minipage}
\end{center}
\caption{All non-empty rich factors $u$ of $w$ from Example~\ref{ex:exx} together with the pairs $(\lpp(u), \lps(u))$.}
\label{ta:exx}
\end{table}

\end{example}

Let us formulate  another sufficient condition for non-richness of a word $w$.

\begin{proposition}\label{my} Let $u$ and $v$ be two words satisfying
\begin{equation}\label{E2} u\neq  \widetilde{v}, \ \ u,v \ \  \text{rich}, \quad    \lps(u)=\lpp(v) \quad \text{and}\quad  \lps(v)=\lpp(u).
\end{equation} If a word $w$ contains factors $u$ and $v$,  then $w$ is not rich.
\end{proposition}

\begin{proof}
First we show (by contradiction)   that  the assumption \eqref{E2} gives
\begin{equation}\label{E3}
u, \widetilde{u} \notin \L(v)\cup \L(\widetilde{v}) \quad \text{and}\quad  v, \widetilde{v} \notin \L(u)\cup \L(\widetilde{u}).
\end{equation}
As the roles of $v$ and $u$ are symmetric,  we have to  discuss the following two cases:

1) $u \in \L(v)$: \\
As $v$ is rich,  $\lps(v)$ is unioccurrent in $v$. Since $\lps(v) = \lpp(u)$, we have that $\lpp(u)$ occurs only as a suffix of $v$.
Since $u \in \L(v)$, necessarily $u = \lpp(u)$ and thus $u$ is a palindrome.
It follows that $u=\lps(u)=\lpp(v)=\lps(v)$.  Richness of $v$ implies that  $\lpp(v)$ and $\lps(v)$ are unioccurrent in $v$ and consequently $v$ is a palindrome satisfying $v=\lpp(v) = u = \widetilde{u}$, which is a contradiction.

2) $\widetilde{u} \in \L(v)$: \\
Since  $\lps(v) = \lps(\widetilde{u}) $ is unioccurrent in $v$, we have that  $\widetilde{u}$ occurs only as  a suffix of  $v$. Similarly,     as  $\lpp(v) = \lpp(\widetilde{u}) $ is unioccurrent in $v$, we get that  $\widetilde{u}$ occurs only as  a prefix  of  $v$. It implies  $v=  \widetilde{u}$, which is again a contradiction.

Obviously,  the assumption \eqref{E2} implies that $u$ and $v$ are not palindromes.
\medskip

To prove the proposition itself (again by contradiction), we assume that  $w$ is rich and let $f$ denote the shortest factor of $w$ such that $f$ contains as its factor $u$ or $\widetilde{u}$  and    $f$ contains as its factor $v$ or $\widetilde{v}$.
Without loss of generality and due to \eqref{E3},  we have to discuss the following two cases:

1) $u$ is a proper prefix and $v$ is a proper suffix of $f$: \\
The word $\lps(f)$ is not longer than $v$; otherwise, we obtain a contradiction with the choice of $f$ as the shortest factor with the given property.   Thus  $\lps(f) = \lps(v)$. Similarly, $\lpp(f) = \lpp(u)$.  It means that $ \lps(f)$ is not unioccurrent in $f$  --- a contradiction.

2) $u$ is a proper prefix and $\widetilde{v}$ is a proper suffix of $f$: \\
By the same argument as before,   $\lps(f) = \lps(\widetilde{v})= \lpp(v)$.
It means that $\lpp(v) = \lps(u)$ occurs as a suffix of $f$ and also as a suffix of $u$. Since $u$ is a proper prefix of $f$,  the factor $\lpp(v) = \lps(f)$ occurs in $f$ twice ---  a contradiction with the richness of $f$.
\end{proof}

\begin{example}  We consider again  the non-rich word $w=11010011$. It contains the factors $u =11010$, $v = 010011$ such that $\lpp(u) =11=\lps(v)$  and   $\lps(u) =010=\lpp(v)$. Also the pairs $u' = 1101001$, $v' = 10011$ and $u'' = 110100$, $v'' = 0011$ satisfy \eqref{E2}.
\end{example}

We show that a pair of factors with the property \eqref{E2} occurs in each  non-rich word.

\begin{proposition}\label{my2} If $w$ be is a non-rich word, then  $w$ has two factors $u$ and $v$ such that
$$ u\neq  \widetilde{v}, \ \ u,v \ \  \text{rich}, \quad    \lps(u)=\lpp(v) \quad \text{ and }\quad  \lps(v)=\lpp(u).
$$
\end{proposition}

\begin{proof}
As $w$ is not rich, it contains a complete  return word $r$ to a palindrome $p$ such that $r$ is not a palindrome.
Let $r$ be the shortest non-palindromic return word in $w$  to a palindrome.   Denote by $t$ the first letter of $r$ and find the longest $q$  such that  $tq$ is  a prefix of $r$ and $\widetilde{q}t$ is a suffix of $v$. Clearly, $p$ is a prefix of $tq$ and $p$ is a suffix  of $\widetilde{q}t$.  Let us denote $x$ and $y$  the letters such that $tqx$ is a prefix of $r$ and $y\widetilde{q}t$. Obviously, $x\neq y$.
\begin{itemize}
\item If $q$ is empty, then $r$ is a non-palindromic complete return word to the letter $t$, i.e., the letter $t$ does not occur in the factor $f$ given by $r = tft$, i.e., $f=t^{-1}rt^{-1}$.     Choose $z \in \{x,y\}$ such that $z\neq t$ and put

$u:=$ the shortest prefix of $r$ which ends with the letter $z$ and

 $v:=$ the shortest suffix of $r$ which starts with the letter $z$.

In particular, both letters  $z$ and $t$ are unioccurrent in $u$ and also in $v$. It means that $\lpp(u)=t=\lps(v)$ and $\lps(u)=z=\lpp(v)$.
One of the words $u$ and $v$ has length 2 and the second one is longer than 2.  It implies that $u\neq \widetilde{v}$.

\item Let us assume that $q\neq \varepsilon$. The word $f= t^{-1}rt^{-1}$  has a prefix $qx$ and a suffix  $y\widetilde{q}$. First we show
 \smallskip

\noindent {\it  Claim}: Occurrences of  $q$  and $\widetilde{q}$ in $f$ alternate and moreover each factor of $f$  starting with $q$ and ending with $\widetilde{q}$ without other occurrences of $q$ and $\widetilde{q}$ is a palindrome.
\smallskip

{\it Proof of the claim}: Let $w'$ be arbitrary  suffix of $f$ such that $|w'| > |\widetilde{q}|$  and  $w'$ has a prefix $q$.  Clearly,  $f$ has a suffix $\widetilde{q}$ and thus  $\widetilde{q}$ is a suffix of  $w'$ as well.  Let us denote $p'=\lpp(q)$.  Since $q$ is rich,     $p'$ is unioccurrent in $q$. But $p'$ occurs in $w'$ at least twice,  as $\widetilde{q}$  is a suffix of $w'$.
Let us denote $r'$ a complete  return word to $p'$ in $w'$.
From minimality of $r$, the complete  return word $r'$ to $p'$ is a palindrome. Therefore, $w'$ has prefixes $p'$,  $q$ and $r'$, their lengths satisfy $|p'|\leq |q| < |r'|$ . It implies that $\widetilde{q}$ is a suffix of the palindrome $r'$ and thus the first occurrence of $q$ in $w'$  is followed by the occurrence of $\widetilde{q}$.

\medskip
Since $f$ is not a palindrome, the previous claim implies that $q$ and $\widetilde{q}$ occur also as inner factors of $f$.
It means that there exists a palindromic factor, say $w''$,  of the word $f$ such that    $\widetilde{q}$ is a prefix and $q$ is a suffix of $w''$  and $|w''| > |q|$.
Let $z$ denote the letter satisfying  that $\widetilde{q}z$  is a prefix  of $w''$. Obviously, $zq$  is a suffix of $w''$.
Let us stress that $z\neq t$, otherwise $r$ would not be a complete return word to the palindrome $p$.
The  letter $z$ enables us to identify the  factors $v$  and $u$ announced in the proposition.   Put

$u:=$ the shortest prefix of $r = tft$ which ends with $\widetilde{q}z$

$v:=$ the shortest suffix of $r = tft$ which starts with $zq$.

To prove  $\lpp(u)=\lps(v)$, we apply the simple observation: If a word $s'$ is a prefix of a word $s$ and $\lpp(s)$ is a prefix of $s'$, then $\lpp(s)=\lpp(s')$.

In our situation:  $p = \lpp(r)=\lpp(u)$. Analogously, $p = \lps(r)=\lps(v)$.

To show  $\lps(u)=\lpp(v)$, we  use a simple consequence of the claim:
Any occurrence of  $\ell q$ in $r$, where $\ell$ is a letter with $\ell \neq t$, is preceded with an occurrence of $\widetilde{q}\ell$. Therefore,  our definition of $u$ guarantees that $\lps(u)$ is not longer than  $\widetilde{q}z$, i.e., $\lps(u) =\lps(\widetilde{q}z)$. By the same reason,
$\lpp(v) =\lpp(zq)$.  As  $\lpp(zq) =  \lps(\widetilde{q}z)$,  the equality $\lps(u)=\lpp(v)$ is proven.

Obviously, $u\neq \widetilde{v}$. Otherwise,  we have a contradiction with the assumption that $tq$ is the longest prefix of $r$ such that $\widetilde{q}t$  is a suffix of $r$. 
\end{itemize}
\end{proof}


The last proof has an interesting direct consequence on a binary alphabet.
It is based on the fact that the case $q = \varepsilon$ is not possible on a binary alphabet and the second case implies that $q$ is not a palindrome.
We state this consequence of the construction in the second case as the following corollary.

\begin{corollary}
Let $w \in \{0,1\}^*$ be a binary word.
The word $w$ is not rich if and only if there exists a non-palindromic word  $q$ such that
\[
0q0, 1q1, 0\widetilde{q}1, 1\widetilde{q}0 \in \L(w).
\]
\end{corollary}

\section{Open Questions and Related Problems}

We finish this article with a list of open questions that we deem important in further understanding of the structure of rich words (and more generally, words with finite palindromic defect).

\begin{itemize}
\item The subexponential upper bound on the number of rich words $R_d(n) $   of length $n$ over $d$ letters is based on  the statement that  any rich word of length $n$ can be factorized into at most $c\frac{n}{\ln n}$ distinct palindromes. In fact, the number of palindromes is exaggerated, as the factorization does not take into consideration that each of the palindromes is rich as well.  Any asymptotic improvement of the bound  $c\frac{n}{\ln n}$ would improve the upper bound on $R_d(n) $.
\item  To our knowledge, there are no results on morphisms preserving the set of rich words.
Such a class of morphisms preserving richness would allow to construct a set of class other than the set constructed in \cite{GuShSh15} to obtain a lower bound on $R_2(n)$.
In particular, any fixed point of a primitive  morphism which preserves the set of rich words must be rich as well.
In this point of view the following question is also important.
\item Theorem \ref{th:0dc_main} confirms the validity of the zero defect conjecture only  for marked morphisms $\varphi$ satisfying the following assumption:
all complete return words of all letters in $\buu$ are palindromes or $\varphi$ is not stationary.   We have no example that this peculiar assumption is really needed.
\item  We do not know how to decide whether two rich words $u$ and $v$  are factors of a common rich word $w$. The related task is to identify   a minimal non-rich word, i.e., to look for a word which is not rich but any its proper factor is rich.
\end{itemize}

Primitive morphisms that preserve the set of rich words are included in a larger set of morphisms having infinitely many palindromic factors in their fixed points.
An infinite word having infinitely many palindromic factors is usually called {\em palindromic}.
A very useful property of morphisms in this larger set is given by the following conjecture.

\begin{conjecture}[Class $P$ conjecture \cite{HoKnSi}] \label{co:HKS}
Let $\buu$ be a palindromic fixed point of a primitive morphism $\varphi$.
There exists a morphism of class $P'$ such that its fixed point has the same language as $\buu$.
\end{conjecture}

The original statement of the conjecture in \cite{HoKnSi} is ambiguous and allows for more interpretations, see also \cite{LaPe14} or \cite{HaVeZa}.
The above given statement of Conjecture~\ref{co:HKS} follows from two results.
First, for binary alphabet the question is solved by B. Tan  in \cite{BoTan}: if a fixed point of a primitive morphism $\varphi$ over a binary alphabet contains infinitely many palindromes, then $\varphi$ or $\varphi^2$ is of class $P'$.
Second, in \cite{La2013}, S. Labb\'e shows that the analogy of the previous result cannot be generalized for  multiliteral alphabet: there exists a word $\bww$ over ternary alphabet which is a palindromic fixed point of a primitive morphism and not being fixed by any morphism of class $P'$.
However, the authors of \cite{HaVeZa} note that the language of the word $\bww$ may indeed be generated by a morphism of class $P$.

At this moment only partial answers to Conjecture~\ref{co:HKS} are known: as already mentioned, the binary case is solved (\cite{BoTan}); for larger alphabets an affirmative answer is provided only for some special classes of morphisms.

In \cite{MaPeSta1}, we confirm the conjecture for morphisms fixing a codings a non-degenerate exchange of $3$ intervals.
In \cite{LaPe14}, the authors prove the validity of the conjecture for marked morphisms.
Moreover, they show that a power of the marked morphism itself is in class $P'$.
The technique and results used in the proofs of the latter fact is crucial in showing the defect conjecture for marked morphisms in \cite{LaPeSta1}.

Palindromicity of a fixed point $\buu$ is linked to the symmetry of the language $\Lu$, namely the closedness under reversal.
One direction of this connection is trivial:
If a fixed point of a primitive morphism contains infinitely many palindromes, then its language is closed under reversal.
The non-trivial converse is shown in \cite{LaPe14} for marked morphisms.
The mentioned results and computer experiments lead to the formulation of the following conjecture.

\begin{conjecture}
Let $\varphi: \A^* \to \A^*$ be a primitive morphism having a fixed point $\buu$.
Its language $\Lu$ is closed under reversal if and only if $\buu$ is palindromic.
\label{co:last}
\end{conjecture}

A proof in full generality of this conjecture has applications in algorithmic analysis of the language of a given morphism.
Specifically, it allows for an efficient test whether the language of a fixed point is closed under reversal.
For marked primitive morphisms, such an algorithm may be devised based on the following results of \cite{LaPe14}:
\begin{enumerate}
    \item Every marked morphism has a so-called well-marked power (see \cite{LaPe14} for a definition). If the fixed point of the morphism is palindromic, then this power is of class $P'$.
    \item Conjecture~\ref{co:last} is true for marked morphisms.
\end{enumerate}
Overall, closedness under reversal of the language generated by a marked primitive morphism is equivalent to palindromicity of the language which is equivalent to the well-marked power being in class $P'$.
Therefore, given a marked primitive morphism, the test whether the language it generates is closed under reversal consists of finding the well-marked power and checking if this power is in class $P'$.
Since both these tasks can be performed efficiently in a straightforward manner, the whole test can be easily executed.

In the view of this special case, Conjecture~\ref{co:last} may be seen as a first step to provide an efficient test of closedness under reversal for the language generated by any primitive morphism for which the class $P$ conjecture holds.


\section*{Acknowledgements}
The authors acknowledge financial support by the Czech Science Foundation grant GA\v CR 13-03538S.




\bibliographystyle{splncs03}
\IfFileExists{biblio.bib}{\bibliography{biblio}}{\bibliography{../../../!bibliography/biblio}}

\end{document}